\documentclass[letterpaper,twoside,twocolumn]{IEEEtran}

\usepackage{multicol}
\usepackage[footnotesize]{caption}
\usepackage{graphicx,float,subfig}
\usepackage{url}
\usepackage{amssymb,amsmath,amsfonts,layout}
\usepackage{makeidx}
\usepackage{blkarray,multirow}
\usepackage{enumerate}
\usepackage{algpseudocode}
\usepackage{algorithm}

\usepackage{amsthm}
\usepackage{tikz}
\usepackage{bbm}
\usepackage{pgfplots}
\pgfplotsset{compat=1.9}
\usepgfplotslibrary{fillbetween}
\usetikzlibrary {positioning}

\algnewcommand{\Inputs}[1]{%
	\State \textbf{Inputs:}
	\Statex \hspace*{\algorithmicindent}\parbox[t]{.8\linewidth}{\raggedright #1}
}
\algnewcommand{\Initialize}[1]{%
	\State \textbf{Initialize:}
	\Statex \hspace*{\algorithmicindent}\parbox[t]{.8\linewidth}{\raggedright #1}
}
\algnewcommand{\Outputs}[1]{%
	\State \textbf{Outputs:}
	\Statex \hspace*{\algorithmicindent}\parbox[t]{.8\linewidth}{\raggedright #1}
}

\usepackage{tikz}
\usetikzlibrary{shapes,arrows,positioning} 
\tikzset{
	>=stealth',
	block/.style={
		rectangle,
		rounded corners,
		draw=black, very thick,
		text width=12em,
		minimum height=3em,
		text centered},
	link/.style={
		->,
		thick,
		shorten <=2pt,
		shorten >=2pt},
	decision/.style={
		diamond,
		draw, very thick,
		fill=blue!20, 
		text width=8em,
		aspect=3,
		text centered}
}

\theoremstyle{plain}

\newtheorem{theorem}{Theorem}[section]
\newtheorem{proposition}[theorem]{Proposition}
\newtheorem{lemma}[theorem]{Lemma}

\newtheorem{remark}[theorem]{Remark}

\newcommand{\longthmtitle}[1]{\mbox{}{\bf \textit{(#1).}}}

\newcommand{\oprocendsymbol}{\hbox{$\square$}}
\newcommand{\oprocend}{\relax\ifmmode\else\unskip\hfill\fi\oprocendsymbol}

\makeatletter
\newcommand{\pushright}[1]{\ifmeasuring@#1\else\omit\hfill$\displaystyle#1$\fi\ignorespaces}
\newcommand{\pushleft}[1]{\ifmeasuring@#1\else\omit$\displaystyle#1$\hfill\fi\ignorespaces}
\renewcommand*\env@matrix[1][*\c@MaxMatrixCols c]{%
	\hskip -\arraycolsep
	\let\@ifnextchar\new@ifnextchar
	\array{#1}}
\makeatother

\DeclareMathOperator{\tr}{\bf Tr}

\DeclareMathOperator{\B}{\mathcal{B}}
\DeclareMathOperator{\Co}{\mathcal{C}}
\DeclareMathOperator{\Pa}{\mathcal{P}}
\DeclareMathOperator{\Q}{\mathcal{Q}}

\DeclareMathOperator{\N}{\mathcal{N}}

\DeclareMathOperator{\V}{\mathcal{V}}
\DeclareMathOperator{\Pint}{\mathbb{N}}
\DeclareMathOperator{\Sn}{\mathbb{S}}

\DeclareMathOperator{\R}{\mathbb{R}}
\DeclareMathOperator{\C}{\mathbb{C}}
\DeclareMathOperator{\E}{\mathcal{E}}

\DeclareMathOperator{\St}{\mathcal{S}}
\DeclareMathOperator{\graph}{\mathcal{G}}

\newcommand{\rank}{\operatorname{rank}}
\newcommand{\opt}{\operatorname{opt}}

\newcommand{\real}{\ensuremath{\mathbb{R}}}
\newcommand{\complex}{\ensuremath{\mathbb{C}}}

\renewcommand{\bar}{\overline}

\newcommand{\setdef}[2]{\{#1 \; | \; #2\}}

\allowdisplaybreaks

\newcommand{\rev}[1]{{\color{blue} #1}}
\renewcommand{\rev}[1]{#1}

\title{\rev{Virtual-Voltage Partition-Based Approach to Optimal
    Transmission Switching}}

\author{Chin-Yao Chang \quad Sonia Mart{\'\i}nez \quad Jorge Cort\'es
  \thanks{C.-Y. Chang, S. Mart{\'\i}nez, and J. Cort\'es are
    with the Department of Mechanical and Aerospace Engineering,
    UC San Diego.
  }
  \thanks{A preliminary version of this work appeared
    as~\cite{CYC-SM-JC:17-allerton} at the 2017 Allerton Conference on
    Communications, Control, and Computing.}  }

\begin{document}
\maketitle

\begin{abstract}
  \rev{This paper deals with optimal transmission switching (OTS)
    problems involving discrete binary decisions about network
    topology and non-convex power flow constraints.} We adopt a
  semidefinite programming formulation for the OPF problem which,
  however, remains nonconvex due to the presence of discrete variables
  and bilinear products between the decision variables. To tackle the
  latter, we introduce a novel physically-inspired, virtual-voltage
  approximation that leads to provable lower and upper bounds on the
  solution of the original problem. To deal with the exponential
  complexity caused by the discrete variables, we introduce a graph
  partition-based algorithm which breaks the problem into several
  parallel mixed-integer subproblems of smaller size. Simulations
  demonstrate the high degree of accuracy and affordable computational
  requirements of our approach.
\end{abstract}

\vspace*{-1.5ex}

\section{Introduction}
\rev{Optimal transmission switching (OTS) is concerned with the
  identification of power grid topologies that minimize generation
  cost while maintaining the secure operation of the grid. This leads
  to a class of non-convex optimization problems with discrete
  decision variables and bilinear constraints.
  Solving OTS has the potential to yield significant benefits in
  efficiency and reliability while respecting security constraints,
  but the complexity of solving such highly non-convex problems makes
  achieving this goal difficult. Our focus here is on developing a
  computationally efficient approach to approximate the solution of
  OTS problems.}

\textit{Literature review:} \rev{Discrete variables appear in optimal
  power flow (OPF) problems in many ways, such as transmission
  switching and post-contingency controls.} The
works~\cite{JM-ME-RA:99,SF-IS-SR:12,HA-SDB-MLS:17} provide surveys on
the solution of OPF problems with discrete variables. Many of the
methods employed to solve OPF problems have been extended to deal with
mixed-integer OPF, e.g., particle swarm
optimization~\cite{MRA-MEE:09,PEOY-JMR-CAC:08} and genetic
algorithms~\cite{AGB-PNB-CEZ-VP:02}. \rev{Transmission line switching
  or network topology reconfiguration commonly serve as corrective
  mechanisms in response to system contingencies
  see~\cite{KWH-SSO-RPO:11,JGR-LJBM:99} and references therein.
  In~\cite{JDF-RR-AC:12,PAR-JMF-AR-MCC:12,SF-JL-AA:17}}, linearized
OPF, also known as DCOPF, is deployed to solve OTS
efficiently. Despite its relative low complexity, DCOPF may lead,
especially in congested systems, to poor solutions that can even
result in voltage collapse~\cite{TP-KWH:12, MS-JDF:14}.
The work~\cite{HH-CC-PVH:17} proposes quadratic convex (QC)
relaxations for the MIP-OPF problem, which provides more accurate
results than DCOPF, while still retaining a fast computation time.
\rev{Recent studies~\cite{HH-CC-PVH:17,JM-MM-JCV:16}} show that
methods based on semidefinite programming (SDP) convex relaxations of
ACOPF may lead to better solutions than DCOPF and QC. However, how to
handle variables for transmission switching in the context of SDP is
challenging and not fully understood. The challenges stem from not
only from the integer-valued nature of these variables, but also from
the presence of bilinear terms involving the product of discrete
decision variables with continuous ones, reflecting the impact on the
physical modeling of the line being connected. The
paper~\cite{JM-MM-JCV:16} uses a lift-and-branch-and-bound procedure
to deal with the SDP formulation of mixed-integer OPF (OTS as a
special case), but has still exponential complexity in the worst case
due to the nature of the branch-and-bound procedure.  The
work~\cite{EB-SA-FP:17} also uses SDP to solve OTS problems, where
bilinear terms associated to line connections are addressed by
assuming certain nominal network topology and bilinear terms of other
discrete decision variables are dealt with using the McCormick
relaxation~\cite{GPM:76}.
	
\textit{Statement of contributions:} \rev{We consider the OTS problem
  and introduce a novel SDP convex relaxation to approximate its
  solution.}  Our contributions are twofold.  First, we introduce a
novel way of dealing with the nonconvexity coming from the presence of
bilinear terms, which we term virtual-voltage approximation. Our
approach is based on introducing virtual-voltage variables for the
terminal nodes of each switchable line and impose
physically-meaningful constraints on them.  We show that this approach
leads to sound bounds on individual flows of the switchable lines, and
provides lower and upper bounds on the optimal value of the original
problem.
%
%
Our second contribution deals with the nonconvexity coming from the
discrete variables. We build on the virtual-voltage approximation to
propose a graph partition-based algorithm that significantly reduces
the computational complexity of solving the original problem.  This
algorithm uses the values of the optimal dual variables from the
virtual-voltage method to define a weighted network graph, which is
then partitioned with a minimum weight edge-cut set. The algorithm
breaks the original network into sub-networks so as to minimize the
correlation between the solutions to the optimization problem on each
sub-network. Finally, the algorithm solves the OTS problem
on each sub-network in parallel and combines them to reconstruct the
solution of the original problem. 
We implement the proposed algorithms on various IEEE standard test
cases, and compare them with available approaches from the literature
to illustrate their superior performance regarding convergence to the
optima and computation time.


\section{Preliminaries}\label{sec:prelim}
This section introduces basic concepts used in the paper\footnote{We
  use the following notation. We denote by $\Pint$, $\real$,
  $\real_+$, and $\complex$ the sets of positive integer, real,
  positive real, and complex numbers, resp.  We denote by $|\N|$ the
  cardinality of~$\N$.  For a complex number $a\in\mathbb{C}$, we let
  $|a|$ and $\angle a$ be the complex modulus and angle of $a$, and
  its real and imaginary parts are $\Re(a)$ and $\Im(a)$. We let
  $\|v\|$ denote the $2$-norm of $v \in \complex^n$. Let $\Sn^n_+
  \subset \complex^{n \times n}$ and $\mathcal{H}^n \subset \Sn^n_+$
  be the set of positive semidefinite and $n$-dimensional Hermitian
  matrices, resp. For $A\in\mathbb{C}^{n\times n}$, we let $A^*$ and
  $\tr\{A\}$ denote its conjugate transpose and trace, resp. We let
  $A(i,k)$ denote the $(i,k)^{\text{th}}$ element of~$A$.}.

\subsubsection{Graph Theory}\label{sec:graph}
We review basic notions of graph theory
following~\cite{FB-JC-SM:08cor}.  A graph is a pair $\graph =
(\N,\E)$, where $\N \subseteq\Pint$ is the set of nodes and
$\E\subseteq\N\times\N$ is the set of edges. A \textit{self-loop} is
an edge that connects a node to itself.  The graph is
\textit{undirected} if $\{i,k\} = \{k,i\} \in\E$.
A \textit{path} is a sequence of nodes such that any two consecutive
nodes correspond to an edge. The graph is \emph{connected} if there
exists a path between any two nodes.
An \textit{orientation} of an undirected graph is an assignment of
exactly one direction to each of its edges. A graph is \textit{simple}
if it does not have self-loops or multiple edges connecting any pair
of nodes. Throughout the paper, we limit our discussion to undirected,
simple graphs.  A \textit{vertex-induced subgraph} of~$\graph$,
written $\graph[\N_s] = (\N_s,\E_s)$, satisfies $\N_s\subseteq\N$ and
$\E_s = \E\cap(\N_s\times\N_s)$. An \textit{edge cut set} is a subset
of edges which, if removed, disconnects the graph. A \textit{weighted
  graph} is a graph where each branch $\{i,k\}$ has a weight,
$w_{ik}\in\real_+$. Given the edge weights $w\in\real_+^{|\E|}$, the
\textit{adjacency matrix} $A$ has $A(i,k) = A(k,i) = w_{ik}$ if
$\{i,k\} \in \E$, and $A(i,k)=0$ otherwise. The \textit{degree matrix}
$D$ is a diagonal matrix such that $D(i,i) =
\sum_{k,\{i,k\}\in\E}w_{ik}$.  \rev{The \textit{normalized adjacency
    matrix} is  $A_n = \sqrt{D}^{-1}A\sqrt{D}^{-1}$.  } The
\textit{Laplacian matrix} is $L = D-A$. The Laplacian matrix is
positive semidefinite, with zero as an eigenvalue and multiplicity
equal to the number of connected components in the graph.  The
\emph{Fiedler vector} is the eigenvector associated with the second
smallest eigenvalue of~$L$. An \emph{$n$-partition} of a connected
$\graph = (\N,\E,A)$ divides $\graph$ into a number of $n$ connected
vertex-induced subgraphs, $\graph[\V_i]$, such that $\cup_{i=1}^n \V_i
= \N$ and $\V_i\cap\V_k=\emptyset$ for all $i\neq k$. An
\emph{$n$-optimal partition} of $\graph = (\N,\E,A)$ is an
$n$-partition of $\graph$ with $\sum_{\{i,k\}\in\E_c }w_{ik}$
minimized, where $\E_c = \E\setminus(\cup_{i=1}^n \V_i\times\V_i )$.
\textit{Spectral graph partitioning} partitions a connected graph
$\graph$ into two vertex-induced subgraphs, $\graph[\N_1]$ and
$\graph[\N_2]$, where $\N_1$ and $\N_2$ are the nodes corresponding to
the positive and non-positive entries of the Fiedler vector,
respectively.

\subsubsection{McCormick Relaxation of Bilinear
  Terms}\label{sec:McCormick}
The McCormick envelopes~\cite{GPM:76} provide linear relaxations for
optimization problems that involve bilinear terms. Consider a bilinear
term $xy$ on the variables $x, y\in\real$, for which there exist upper
and lower bounds, $ \underline{x}\le x\le \bar{x}, \quad
\underline{y}\le y\le \bar{y}$.  The McCormick relaxation consists of
substituting in the optimization problem the term $xy$ by its
surrogate $v\in\real$ and adding the following McCormick envelopes
on~$v$,
\begin{subequations}\label{eq:McCormick}
  \begin{align}
    &v \ge \underline{x}y + x\underline{y} -
    \underline{x}\underline{y}, \quad v \ge \bar{x}y + x\bar{y} -
    \bar{x}\bar{y},
    \\
    &v \le \bar{x}y + x\underline{y} - \bar{x}\underline{y}, \quad v
    \le x\bar{y} + \underline{x}y - \underline{x}\bar{y}.
  \end{align}
\end{subequations}
Constraints~\eqref{eq:McCormick} are tight, in the sense that each
plane in~\eqref{eq:McCormick} is tangent to the bilinear-constraint
manifold at two boundary lines.  The convex polyhedron in the
variables $(x,y,v)$ encloses the actual bilinear-constraint manifold.
	
\section{Problem Statement}\label{sec: ProbSet}
We begin with the formulation of the OPF problem over an electrical
network and its SDP convex relaxation following~\cite{JL-SHL:12}.
Then, we introduce binary variables leading to the \rev{OTS} problem
formulation of interest in this paper.

Consider an electrical network with generation buses~$\N_G$, load
buses~$\N_L$, and electrical interconnections described by an
undirected edge set $\E_0$.  Let $\N= \N_G \cup \N_L$ and denote its
cardinality by~$N$.  We denote the phasor voltage at bus $i$ by $V_i =
E_i e^{j \theta_i}$, where $E_i\in\mathbb{R}$ and
$\theta_i\in[-\pi,\pi)$ are the voltage magnitude and phase angle,
respectively.  For convenience, $V = \setdef{V_i}{i\in\N}$ denotes the
collection of voltages at all buses. The active and reactive power
injections at bus $i$ are given by the power flow equations
\begin{align}\label{eq:PFE}
  P_i & = \tr\{Y_iVV^*\} + P_{D_i}, \quad Q_i =
  \tr\{\bar{Y}_iVV^*\} + Q_{D_i},
\end{align}
where $P_{D_i},Q_{D_i}\in\R$ are the active and reactive power demands
at bus~$i$, 
and $Y_i, \bar{Y}_i \in\mathcal{H}^{N}$ are derived from the
admittance matrix \textbf{Y}$\in \complex^{N \times N}$ as follows
\begin{subequations}\label{eq:admit}
  \begin{align}
    Y_i & = \frac{(e_ie_i^\top\textbf{Y})^* + e_ie_i^\top\textbf{Y}}{2},\quad 
    \bar{Y}_i = \frac{(e_ie_i^\top\textbf{Y})^* -
      e_ie_i^\top\textbf{Y}}{2j}.
  \end{align}  
\end{subequations}
Here $\{e_i\}_{i= 1,\dots,N}$ denotes the canonical basis of
$\real^N$.  The OPF problem also involves the box constraints
\begin{align}\label{eq:allConst}
  &\underline{V}_i^2 \leq |V_i|^2 \leq \bar{V}_i^2, \; \forall i\in
  \mathcal{N}, \nonumber
  \\
  &\underline{P}_i \leq P_i \leq \bar{P}_i, \;\; \underline{Q}_i \leq
  Q_i \leq \bar{Q}_i, \; \forall i\in
  \mathcal{N},
  \\
  \nonumber &|V_{i} - V_{k}|^2 \leq \bar{V}_{ik}, \;\forall \{i,k\}
  \in \E_0,
\end{align}
where $\bar{V}_{ik}$ is the upper bound of the voltage difference
between buses $i,k$, and $\underline{V}_i$ and $\bar{V}_i$ are the
lower and upper bounds of the voltage magnitude at bus $i$,
respectively.  \rev{Notice that the upper bound of the voltage
  difference is equivalent to line thermal constraints, as described
  for instance in~\cite{RM-MA-JL:16}. The voltage difference
  constraints prevent overheating of transmission lines.}  All
$\underline{P}_i,\underline{Q}_i, \bar{P}_i, \bar{Q}_i$, are defined
similarly. The objective function is typically given as a quadratic
function of the active power,
\begin{align}\label{eq:Cost}
  & \sum\nolimits_{k \in \mathcal{N}_G}c_{i2} P_i^2 + c_{i1} P_i,
\end{align}
where $c_{i2}\geq 0$, and $c_{i1}\in \mathbb{R}$. The OPF problem is
the minimization over~\eqref{eq:Cost} subject to~\eqref{eq:PFE}
and~\eqref{eq:allConst}. Such optimization is non-convex due to the
quadratic terms on $V$.  To address this, one can equivalently define
$W = VV^*\in \mathcal{H}^N$ (or $W\in \mathcal{H}^N$ and $\rank(W)=1$)
as the decision variable (note all the terms in~\eqref{eq:PFE},
\eqref{eq:allConst} and~\eqref{eq:Cost} are quadratic in
$V$). Dropping the rank constraint on $W$ makes the OPF problem
convex, giving rise to the SDP convex relaxation,
\begin{flalign*}
  \textbf{(P1)} \hspace{15mm}\min_{W\succeq 0} \sum\nolimits_{i \in
    \N_G} c_{i2}
  P_i^2+
  c_{i1} P_i, &&
\end{flalign*}
subject to
\begin{subequations}\label{eq:const}
  \begin{alignat}{7}
    & P_i = \tr\{Y_iW\} + P_{D_i}, \;
    \forall i\in \mathcal{N}, \label{eq:const-a}\\
    & Q_i = \tr\{\bar{Y}_iW\} + Q_{D_i}, \;
    \forall i\in \mathcal{N}, \label{eq:const-b}\\
    & \underline{P}_i \leq P_i \leq \bar{P}_i, \;\underline{Q}_i \leq Q_i \leq \bar{Q}_i
    , \;\;
    \forall i\in \mathcal{N},
    \label{eq:const-c}
    \\
    & \underline{V}_i^2 \leq \tr\{M_iW\} \leq \bar{V}_i^2, \; \forall
    i\in \mathcal{N} , \displaybreak[0]
    \label{eq:const-e}
    \\
    & \tr\{M_{ik}W\} \leq \bar{V}_{ik}, \; \forall \{i,k\} \in
    \E_0,
    \label{eq:const-f}
  \end{alignat}
\end{subequations}
where $M_i, M_{ik}\in \mathcal{H}^{N}$ are defined so that
$\tr\{M_iW\} = |V_i|^2$ and $\tr\{M_{ik}W\} = |V_i-V_k|^2$.

\rev{In the OTS problem, the set of transmission lines $\E_0$ is
  divided into a set of switchable $\E_s$ and non-switchable $\E$
  lines such that $ \E_0 = \E_s\cup\E$ (note that $\E_s = \E_0$ and
  $\E=\emptyset$ is possible).  Choosing among the switchable lines
  which ones are active affects the nodal active and reactive power
  injections in~\eqref{eq:const-c}.  The question is then to determine
  what the optimal choice of switching lines is. We formalize this
  problem next.}  For each line $\{i,k\}\in\E_s$, we define a binary
variable $\alpha_{ik}\in\{0,1\}$, and we say the line is connected if
$\alpha_{ik}=1$ and disconnected otherwise. If $\alpha_{ik}=1$, then
the power flow from node $i$ to $k$ through edge $\{i,k\}\in\E_s$ is
\begin{align}\label{eq:swPQ}
  P_{ik} = \tr\{Y_{ik}W\}, \quad Q_{ik} = \tr\{\bar{Y}_{ik}W\},
\end{align}
where $Y_{ik},\bar{Y}_{ik} \in \C^{N\times N}$ are defined as follows:
all entries are prescribed to be zero except the ones defined
by\footnote{We omit charging susceptance, tap ratio and phase shift of
  transformers for simplicity.}
\begin{align*}
  Y_{ik}(i,i) &= \Re(y_{ik}), \quad  Y_{ik}(i,k) = Y^*_{ik}(i,k)= -y_{ik}/2,\\
  \bar{Y}_{ik}(i,i) &= \Im(y_{ik}), \quad \bar{Y}_{ik}(i,k) =
  \bar{Y}^*_{ik}(i,k) =-j\cdot y_{ik}/2.
\end{align*}
Here $y_{ik}\in\complex$ is the admittance of
line~$\{i,k\}$. Taking~\eqref{eq:swPQ} into account, the active and
reactive power of each node become
\begin{align}\label{eq:nodal_s_power}
  &P_i = \tr\{Y_iW\} + P_{D_i}+ \sum\nolimits_{k\in\St_{i}}\alpha_{ik}P_{ik},
  \\
  \nonumber &Q_i = \tr\{\bar{Y}_iW\} + Q_{D_i}+
  \sum\nolimits_{k\in\St_{i}}\alpha_{ik}Q_{ik},
\end{align}
where $\St_{i}:= \setdef{k}{\{i,k\}\in\E_s}$.
%
Given $\alpha\in\{0,1\}^{|{\E}_s|}$, we use \textbf{(P1)}-$\alpha$ to
refer to the OPF problem solved with the network topology with extra
lines as determined by $\alpha$.

We are interested in solving what we call \textbf{(P2)}, which is the
optimization \textbf{(P1)} with constraints~\eqref{eq:const-a} and
\eqref{eq:const-b} replaced
by~\eqref{eq:swPQ}~and~\eqref{eq:nodal_s_power}.  \rev{In addition to
  the optimization of the power flow, this formulation incorporates
  the optimal choice among the switchable lines in $\E_s$.  }  The
problem~\textbf{(P2)} is non-convex for two reasons: the binary
variables $\alpha_{ik}$ and the bilinear products of $\alpha_{ik}$
and~$W$. The first problem can be addressed using existing
integer-programming solvers~\cite{RAJ-RS-BCP:12,EB-SA-FP:17}. The
McCormick relaxation described in Section~\ref{sec:McCormick} is the
standard way to deal with the second problem. In this paper, we
instead provide alternative routes to address each of these problems
for the optimization~\textbf{(P2)}.

\rev{
  \begin{remark}\longthmtitle{Networks where all lines are
      switchable}\label{rem:allswline}
    {\rm General formulations of the OTS problem assume all
      transmission lines are switchable, see
      e.g.~\cite{PAR-AR-MCC-EG-EN-CRP:12,HH-CC-PVH:17}. Some works
      also pre-select a pool of switchable lines through heuristic
      methods~\cite{PAR-AR-MCC-EG-EN-CRP:12}. 
      The approach described next is applicable to both types of
      scenarios.
    }
    \oprocend
  \end{remark}
}

	
\section{Virtual-Voltage Approximation of
  Bilinear~Terms}\label{sec:bilinear}

We introduce here a novel way to deal with the bilinear terms
in~\textbf{(P2)} which we term \emph{virtual-voltage
  approximation}. We start by noting that every binary variable
$\alpha$ only appears in the bilinear products
in~\eqref{eq:nodal_s_power} together with another continuous
variable~$W$. If we convexify the binary variables by having them take
values in $[0,1]$, then we can interpret each bilinear term
corresponding to $\{i,k\}\in\E_s$ as a line power flow from $i$ to
$k$, with the magnitude bounded by what $W$ indicates. Following this
reasoning, if the direction of power flow of every line
$\{i,k\}\in\E_s$ was known, then the bilinear term would no longer be
an issue. For example, if we knew that $P_{ik} = \tr\{Y_{ik}W\}
\in\real_+$ and $Q_{ik} = \tr\{\bar{Y}_{ik}W\} \in\real_+$, then we
could define new variables, $\hat{P}_{ik}\in\real$ and
$\hat{Q}_{ik}\in\real$, replacing $\alpha_{ik}P_{ik}$ and
$\alpha_{ik}Q_{ik}$ in~\eqref{eq:nodal_s_power}, respectively, and
impose
\begin{align}\label{eq:knownLP}
  0 \leq \hat{P}_{ik}\leq P_{ik}, \quad 0 \leq \hat{Q}_{ik}\leq
  Q_{ik}.
\end{align}
This would eliminate the bilinear terms and the only remaining
non-convexity would be that the physical feasible solution should
satisfy $\hat{P}_{ik} \in \{0, P_{ik}\}$ and $\hat{Q}_{ik} \in \{0,
Q_{ik}\}$.  In general, however, the direction of power flow of the
lines $\{i,k\}\in\E_s$ is not known a priori and, hence, the trivial
convex constraints~\eqref{eq:knownLP} for the relaxation are no longer
valid.

\subsection{Convex Relaxation Via Virtual
  Voltages}\label{sec:convex-relaxation-virtual}

Our idea to approximate each bilinear term builds on defining a
virtual-voltage for the terminal nodes of the line and impose
constraints on them to make sure they have physical sense.  We make
this precise next. Let $\hat{\E}_s$ be an arbitrary orientation of
$\E_s$.  To define the virtual-voltages, and in keeping with the SDP
approach, for each $\{i,k\}\in\hat{\E}_s$ we introduce a two-by-two
positive semidefinite matrix $U_{ik} \in \Sn^2_+$. This matrix encodes
physically meaningful voltages at the terminal nodes if its rank is
one, namely, $U_{ik} = u_{ik} u_{ik}^\top$, with $u_{ik}(1)$ and
$u_{ik}(2)$ corresponding to the voltages of nodes~$i$ and~$k$,
respectively.  For convenience, we introduce $\hat{M} = [
\begin{smallmatrix}
  1 & -1
  \\
  -1 & 1
\end{smallmatrix}
] $ and impose the following constraints on~$U_{ik}$
\begin{subequations}\label{eq:U_con}
  \begin{alignat}{3}
    U_{ik}(1,1) & \leq \tr\{M_iW\}, \label{eq:U_con1}\\
    U_{ik}(2,2) & \leq \tr\{M_kW\}, \label{eq:U_con2}\\
    \tr\{\hat{M}U_{ik} \} & \leq \tr\{M_{ik}W
    \} . \label{eq:U_con3}
  \end{alignat}
\end{subequations}
Constraints~\eqref{eq:U_con1} and \eqref{eq:U_con2} ensure that the
voltage magnitudes of $i$ and $k$ derived from $U_{ik}$ are no bigger
than the ones from $W$.  Constraint~\eqref{eq:U_con3} ensures that the
voltage difference between nodes $i$ and $k$ computed from $U_{ik}$ is
less than the corresponding difference from $W$. Therefore, if the
matrix $U_{ik}$ has rank one, constraints~\eqref{eq:U_con} ensure that
we obtain physically meaningful and feasible voltage values.

Let $\hat{Y}_{ik}\in\mathbb{C}^{2\times 2}$ be the principal
sub-matrix of $Y_{ik}$ by only keeping the rows and columns associated
with nodes~$i$ and~$k$. We define $\hat{\bar{Y}}_{ik}$ similarly.  We
replace $\alpha_{ik}P_{ik}$ and $\alpha_{ik}Q_{ik}$
in~\eqref{eq:nodal_s_power} by $\tr\{\hat{Y}_{ik}U_{ik} \}$ and
$\tr\{\hat{\bar{Y}}_{ik}U_{ik}\}$, respectively. We now have all the
elements necessary to convexify \textbf{(P2)} as follows
\begin{flalign*}
  \textbf{(P3)} \hspace{5mm}\min_{\substack{W\succeq 0,
      U_{ik}\succeq 0\; \forall \{i,k\}\in\hat\E_s}} \sum\nolimits_{i \in \N_G}
  \Big(c_{i2} P_i^2+ c_{i1} P_i\Big), &&
\end{flalign*}
subject to~\eqref{eq:const-c}-\eqref{eq:const-f},~\eqref{eq:U_con},
and
\begin{subequations}\label{eq:conP3}
  \begin{alignat}{2}
    &P_i = \tr\{Y_iW\} + P_{D_i}+
    \sum\nolimits_{k\in\St_{i}}\tr\{\hat{Y}_{ik}U_{ik} \},\label{eq:conP3-4}
    \\
    &Q_i = \tr\{\bar{Y}_iW\} + Q_{D_i}+
    \sum\nolimits_{k\in\St_{i}}\tr\{\hat{\bar{Y}}_{ik}U_{ik}
    \}. \label{eq:conP3-5}
  \end{alignat}
\end{subequations}


Each optimal solution $U_{ik}^{\opt_3}$ of \textbf{(P3)} has a
dominant eigenvalue, much larger than the other one.  To formally
state the result, let $W^{\opt_3}_{ik} \in\mathcal{H}^2$ denote the
principal sub-matrix of the optimum $W^{\opt_3}$ of \textbf{(P3)}
obtained by removing from $W^{\opt_3}$ all columns and rows except the
ones corresponding to $i$ and~$k$. We use the spectral decomposition
to rewrite $U_{ik}^{\opt_3}$~as
\begin{align*}
  U_{ik}^{\opt_3} = 
    a_{ik}[u_{i}, u_{k}]^\top
  [u_{i}^*, u_{k}^*] + [u_{i}, -u_{k}]^\top [u_{i}^*, -u_{k}^*],
\end{align*}
where $u_{i}\in\complex$, $u_{k}\in\complex$, and $a_{ik}\geq 1$ is
the condition number of $U_{ik}^{\opt_3}$.  Lemma~\ref{lem:rank1-U}
establishes a useful lower bound on~$a_{ik}$.

\begin{lemma}\longthmtitle{Lower bound on the condition
    number}\label{lem:rank1-U}
  For all $\{i,k\}\in\hat\E_s$, the optima of \textbf{(P3)} has
  \begin{align}\label{eq:bdd_a}
     a_{ik}\geq\frac{|u_{i}|^2 + |u_{k}|^2 + 2\Re(u_{i}u_{k}^*)}{\bar{V}_{ik} -
      (|u_{i}|^2 + |u_{k}|^2 - 2\Re(u_{i}u_{k}^*))}.
  \end{align}
\end{lemma}
\begin{proof}
  By constraint~\eqref{eq:U_con3} and \eqref{eq:const-f}, we have
  \begin{align*}
    &(a_{ik}+1)(|u_{i}|^2 + |u_{k}|^2)- 2(a_{ik}-1)\Re(u_{i}u_{k}^*)
    \\
    & \hspace{25mm} \leq
    a_{ik}\tr\{M_{ik}W^{\opt_3}\} \leq a_{ik} \bar{V}_{ik}
  \end{align*}
  Lemma~\ref{lem:rank1-U} follows \rev{by rearranging the terms of the
    first line and the RHS of the second line in the inequality above.
  }
\end{proof}

For all practical purposes, the result of Lemma~\ref{lem:rank1-U}
implies that the matrix $U_{ik}^{\opt_3}$ specifies well-defined
virtual-voltages at the terminal nodes, as we explain next.

\begin{remark}\longthmtitle{Optimal solutions have well-defined virtual
    voltages}\label{rem:rank1-U}
  {\rm Using Lemma~\ref{lem:rank1-U}, we justify that the optimal
    solution $U_{ik}^{\opt_3}$ has a dominant eigenvalue as
    follows. The denominator of~\eqref{eq:bdd_a} is always
    non-negative due to~\eqref{eq:U_con3}. The order of the
    denominator of~\eqref{eq:bdd_a} is at most $10^{-2}$ as
    $\bar{V}_{ik}\approx 10^{-2}$ in most test cases. On the other
    hand, when the virtual-voltage satisfies
    $\tr\{\hat{M}U_{ik}^{\opt_3}\} \approx \tr\{M_{i}W^{\opt_3}\}$ or
    $\tr\{\hat{M}U_{ik}^{\opt_3}\} \approx \tr\{M_{k}W^{\opt_3}\}$,
    then the numerator is lower bounded by a scalar close to one, as
    $\underline{V}_i\approx 1$. As a consequence, the fraction
    in~\eqref{eq:bdd_a} is usually bigger than $10^{2}$.
    Our simulations on IEEE 118 and 300 bus text cases confirm that
    $a_{ik}$ is at least $100$.  }\oprocend
\end{remark}

\subsection{Physical Properties of the Convex Relaxation}

The active and reactive power flows in~\textbf{(P3)} on a switchable
line $\{i,k\}\in\E_s$ are determined by~$U_{ik}$ according to
\begin{align}\label{eq:aux}
  P^{\opt_3}_{ik} = \tr\{\hat{Y}_{ik}U_{ik}^{\opt_3}\},\quad
  Q^{\opt_3}_{ik}=\tr\{\hat{\bar{Y}}_{ik}U_{ik}^{\opt_3}\}.
\end{align}
The next result shows that the optimal power losses on each edge are
bounded by the ones computed from~$W^{\opt_3}$.

\begin{lemma}\longthmtitle{Bounds on the sums of line active and
    reactive powers}\label{lem:UpBdPQ}
  If the line charging susceptance is zero for all
  $\{i,k\}\in\hat\E_s$, then the following inequalities hold
  \begin{subequations}\label{eq:UpBdPQ}
    \begin{align}
      &0\leq P^{\opt_3}_{ik} + P^{\opt_3}_{ki} \leq \tr\{(Y_{ik} +
      Y_{ki})W_{ik}^{\opt_3}\} ,
      \\
      &0\leq Q^{\opt_3}_{ik} + Q^{\opt_3}_{ki}\leq \tr\{(\bar{Y}_{ik} +
      \bar{Y}_{ki})W_{ik}^{\opt_3}\} .
    \end{align}    
  \end{subequations}
\end{lemma}
\begin{proof}
  If the line charging susceptance is zero, then $Y_{ik} + Y_{ki}$
  and $\bar{Y}_{ik} + \bar{Y}_{ki}$ take the following form
  \begin{align*}
    Y_{ik} + Y_{ki} = \hat{M}\Re(y_{ik}), \;
    \bar{Y}_{ik}\hspace{-1mm} + \bar{Y}_{ki} =
    \hspace{-1mm}\hat{M}\Im(-y_{ik}).
  \end{align*}
  Since both $\Re(y_{ik})$ and $\Im(-y_{ik})$ are non-negative, the result
  of Lemma~\ref{lem:UpBdPQ} follows from~\eqref{eq:U_con3} and the
  equalities~\eqref{eq:aux}.
\end{proof}

We next seek to upper bound the individual flows $|P_{ik}|$ and
$|P_{ki}|$. Our next result shows that, under certain conditions for
\textbf{(P3)}, stronger properties hold on the active power retrieved
from the optimal solution $U_{ik}^{\opt_3}$ and~$W^{\opt_3}$.

\begin{proposition}\longthmtitle{Bounds on directional power
    flow}\label{prop:UpBdPQr}
  \rev{Let $w_i = \sqrt{W^{opt_3}(i,i)}$ for each $i \in \N$.  Assume
    $\{i,k\}\in\hat\E_s$ is purely inductive, $|u_k|\in\{0,w_k\}$} and
  \begin{align}\label{eq:volt-diff}
    w_i\geq w_k/2 ,\quad w_k\geq w_i/2.
  \end{align}
  Then the following inequalities hold
  \begin{align}\label{eq:UpBdPQr}
    \!  |P^{\opt_3}_{ik}| \!\leq\! |\tr\{Y_{ik}W_{ik}^{\opt_3}\}| ,\,
    |P^{\opt_3}_{ki}| \!\leq\!  |\tr\{Y_{ki}W_{ik}^{\opt_3}\}|.
  \end{align}
\end{proposition}
\begin{proof}
  If $\{i,k\}$ is purely inductive,
  then
  \begin{align*}
    \hat{Y}_{ik} =\frac{1}{2}
    \begin{bmatrix}
      0 & y_{ik}^*
      \\
      y_{ik} & 0
    \end{bmatrix}, \quad \hat{Y}_{ki} =\frac{1}{2}
    \begin{bmatrix}
      0 & y_{ik}
      \\
      y_{ik}^* & 0
    \end{bmatrix}.
  \end{align*}  
  Note that only the off-diagonal entries of $\hat{Y}_{ik}$ and
  $\hat{Y}_{ki}$ are non-zero, making $P_{ik}^{\opt_3} =
  -P_{ki}^{\opt_3}$ and $|\tr\{Y_{ik}W_{ik}^{\opt_3}\}| =
  |\tr\{Y_{ki}W_{ik}^{\opt_3}\}|$. If $|u_k|=0$,~\eqref{eq:UpBdPQr}
  follows as $P_{ik}^{\opt_3} = -P_{ki}^{\opt_3} = 0$. It is then
  enough to show that if $|u_k|=w_k$, $|P_{ik}^{\opt_3}|\leq
  |\tr\{Y_{ik}W_{ik}^{\opt_3}\}|$. We show it by contradiction. If
  $|P^{\opt_3}_{ik}|>|\tr\{Y_{ik}W_{ik}^{\opt_3}\}|$, then $|u_i|>0 $
  and $ |u_i||u_k||\sin(\theta_{ik}^u)|> w_i
  w_k|\sin(\theta_{ik}^w)|$, where $\theta_{ik}^w = \angle
  W^{\opt_3}_{ik}$, and $\theta_{ik}^u$ is the angle difference
  between $u_i$ and~$u_k$. Using~\eqref{eq:U_con1}, we define
  $\xi_i\geq 1$ such that $\xi_i |u_i| = w_i$, and rewrite the
  inequality as $|\sin(\theta_{ik}^u)|> \xi_i|\sin(\theta_{ik}^w)|$.
  Then,
  \begin{align}\label{eq:contra-1}
    |\cos(\theta_{ik}^u)|< \sqrt{1-\xi_i^2\sin^2(\theta_{ik}^w)}.
  \end{align}
  Rewriting~\eqref{eq:U_con3} as a function of $w_i,w_k,\xi_i$,
  \begin{align}\label{eq:contra-2}
    & 0\leq w_i^2-\frac{w_i^2}{\xi_i^2}
    -2w_iw_k\Big(\cos(\theta_{ik}^w)-\frac{1}{\xi_i}\cos(\theta_{ik}^u)\Big),
  \end{align}
  where we use $|u_k|=w_k$ in the inequality.
  Using~\eqref{eq:contra-1}, the RHS of~\eqref{eq:contra-2} is less
  than
  \begin{align}\label{eq:contra-2r}
    &w_i^2\hspace{-1mm}-\frac{w_i^2}{\xi_i^2} -2w_iw_k \Big(\hspace{-1mm} \cos(\theta_{ik}^w)-\frac{1}{\xi_i}\sqrt{1 -
      \xi_i^2\sin^2(\theta_{ik}^w)}\Big).
  \end{align}
  The derivative~\eqref{eq:contra-2r} with respect to $\xi_i$ is
  \begin{align}\label{eq:derivative}
    &\frac{2 w_i^2}{\xi_i^3}-\frac{4w_iw_k}{\sqrt{1\hspace{-1mm} -
        \hspace{-0.5mm}\xi_i^2\sin^2(\theta_{ik}^w)} } -\frac{2 w_i
      w_k}{\xi_i^2}\sqrt{1\hspace{-1mm} -\hspace{-0.5mm}
      \xi_i^2\sin^2(\theta_{ik}^w)} .
  \end{align}
  The first two elements summed up to a non-positive value due
  to~\eqref{eq:volt-diff}. We then conclude that~\eqref{eq:derivative}
  is non-positive with $W^{\opt_3}_{ik}$ given and
  fixed. Therefore,~\eqref{eq:contra-2r} is non-positive for all
  $\xi_i$ because it is zero when $\xi_i=1$ and is
  non-increasing. But~\eqref{eq:contra-2r} is strictly larger than the
  RHS of~\eqref{eq:contra-2},
  contradicting~\eqref{eq:contra-2}. 
\end{proof}

Condition~\eqref{eq:volt-diff} holds for most existing power
systems~\cite{AJW-BFW:12}.  An analogous result holds by
restricting~$u_i$ instead.

\begin{proposition}\longthmtitle{Bounds on directional power
    flow. II}\label{prop:UpBdPQr-II}
  If $\{i,k\}\in\hat\E_s$ is purely inductive, $|u_i|\in\{0,w_i\}$ and $\bar{V}_{ik}$ is sufficiently
  small such that~\eqref{eq:volt-diff} holds, then~\eqref{eq:UpBdPQr}
  follows.
\end{proposition}

The proof of Proposition~\ref{prop:UpBdPQr-II} is analogous to that of
Proposition~\ref{prop:UpBdPQr} and therefore we omit it. Similar
bounds as~\eqref{eq:UpBdPQr} follow for reactive power if the sum of
the cosine terms in the bracket of~\eqref{eq:contra-2} is non-negative
(however, in general, this is not the case). In addition, more
involved, inequalities as~\eqref{eq:UpBdPQr} hold for the general
impedance case, but we do not pursue them
here. Propositions~\ref{prop:UpBdPQr} and~\ref{prop:UpBdPQr-II} show
that, when the diagonal entries of $U_{ik}^{\opt_3}$ are at the
boundary points of their constraints, \textbf{(P3)} eliminates the
bilinear terms on the active line power flow of \textbf{(P2)} in the
same way as~\eqref{eq:knownLP}. The difference between the relaxations
is that there is no need in \textbf{(P3)} to know the direction of
line power flow a priori, as opposed to~\eqref{eq:knownLP}.

\subsection{Reconstructed Solution to the OTS Problem}
We note that the ratio of the voltage magnitudes derived from
$U_{ik}^{\opt_3}$ and $W^{\opt_3}$ provides an approximation of the
discrete variables $\alpha_{ik}$ in~\eqref{eq:nodal_s_power} as
\begin{align}\label{eq:alphas-from-P3}
  \hat\alpha_{ik} = \tr\{U^{\opt_3}_{ik}\}/\tr\{W^{\opt_3}_{ik}\}.
\end{align}
Note that $\hat\alpha \in [0,1]^{|\hat{\E}_s|}$ because
of~\eqref{eq:U_con}. If we round the entries of $\hat\alpha$ to the
closest number in $\{0,1\}$, we obtain a candidate solution
$\hat\alpha_r \in \{0,1\}^{|\hat{\E}_s|}$ to~\textbf{(P2)}. The
following result, whose proof is straightforward, states the
relationship between \textbf{(P2)} and \textbf{(P3)} based on the
rounded solution~$\hat\alpha_r$.

\begin{proposition}\longthmtitle{Properties of the reconstructed
    solution}\label{prop:reconstructed-sol}
  The optimal values of \textbf{(P1)}-$\hat\alpha_r$, \textbf{(P2)},
  and \textbf{(P3)} satisfy $p_1^{\opt} \ge p_2^{\opt} \ge
  p_3^{\opt}$. Moreover, if $p_1^{\opt}= p_3^{\opt}$, then the optimal
  solution of \textbf{(P1)-$\hat\alpha_r$}, $W^{\opt}_1$, combined
  with $\hat\alpha_r$, is an optimal solution of \textbf{(P2)}.
\end{proposition}

Note that even if $\hat \alpha = \hat \alpha_r
\in\{0,1\}^{|\hat{\E}_s|}$, $p_3^{\opt}$ does not necessary
equal~$p_2^{\opt}$. The reason is that~\eqref{eq:alphas-from-P3}
computes~$\hat \alpha_{ik}$ from the diagonal of~$U^{\opt_3}_{ik}$,
and hence we can not conclude any equality for the off-diagonal
elements of $U^{\opt_3}_{ik}$ and $W^{\opt_3}_{ik}$.  Hence, even if
$\hat \alpha\in\{0,1\}^{|\hat{\E}_s|}$, the optimal solution of
\textbf{(P3)} does not necessarily lie in the feasible region of
\textbf{(P2)}.


\begin{remark}\longthmtitle{Comparison with the McCormick
    relaxation}\label{rem:McC}
  {\rm We explain how we implement the McCormick relaxation on the
    problem~\textbf{(P2)} for comparison purposes.  For each
    $\{i,k\}\in\hat{\E}_s$, we define new variables $\hat{P}_{ik},
    \hat{Q}_{ik}\in\real$ to substitute the bilinear terms
    $\alpha_{ik}P_{ik}, \alpha_{ik}Q_{ik}$, respectively. Then, we
    impose constraints of the form~\eqref{eq:McCormick} on the new
    variables based on $\alpha_{ik}\in\{0,1\}$ and upper and lower
    bounds of active/reactive line power flow,
    $\bar{P}_{ik},\bar{Q}_{ik}\in\real_+$, $\underline{P}_{ik} =
    -\bar{P}_{ik},\underline{Q}_{ik}=-\bar{Q}_{ik}$. 
    If these bounds are far from the actual optimal line power flows,
    this can significantly affect the quality of the solution obtained
    by the McCormick relaxation, a point that we illustrate later in
    our simulations, along with rationale for how to select them.  In
    contrast, the proposed relaxation \textbf{(P3)} is not sensitive
    to those line power bounds, as the virtual-voltages are bounded by
    the power computed from $W$.
    Additionally, the variables $\hat{P}_{ik}$ and $\hat{Q}_{ik}$ in
    the McCormick relaxation are loosely tied to the decision
    variable~$W$, whereas \textbf{(P3)} introduces
    constraints~\eqref{eq:U_con1}-\eqref{eq:U_con3} enforcing a
    stronger physical connection between the virtual-voltages and~$W$. 
    McCormick relaxation 
  } \oprocend
\end{remark}

\rev{
	
  \subsection{$N$-1 Security Constraints}\label{sec:SCOTS}
  The virtual-voltage approximation approach described above can also
  be applied to scenarios with security constraints.  $N$-1 security
  formulations are widely considered in the
  literature~\cite{EBF-RPO-MCF:08,KWH-RPO-EBF-SSO:09,MK-HG-MD:13} and require that,
  under any single component outage (most commonly failures of a
  generator or a transmission line), power flow constraints remain
  satisfied. This prevents cascading failures and makes it easier to
  implement post-contingency controls. The proposed virtual-voltage
  approach can easily accommodate $N$-1 security constraints as we
  explain next. Let $\Co:=\{0,1,\cdots,c\}$ be the set that enumerates
  contingency scenarios, with index $t =0$ corresponding to the
  nominal operating scenario. We formulate the virtual-voltage
  convexifed OTS with $N$-1 security constraints as follows
  \begin{flalign*}
    \centering
    \min_{\substack{W^{[t]}\succeq 0, t\in\Co, \\
        U_{ik}\succeq 0\; \forall \{i,k\}\in\hat\E_s}}
    \sum\nolimits_{i \in \N_G} \Big(c_{i2} P_{i0}^2+ c_{i1}
    P_{i0}\Big), &&
  \end{flalign*}
  subject to the following for all $t\in\Co$
  \begin{subequations}
    \begin{alignat*}{2}
      &U_{ik}(1,1) \leq \tr\{M_iW^{[t]}\}, \;\forall \{i,k\} \in
      \mathcal{E}_s
      \\
      &U_{ik}(2,2) \leq \tr\{M_kW^{[t]}\}, \;\forall \{i,k\} \in
      \mathcal{E}_s
      \\
      &\tr\{\hat{M}U_{ik} \} \leq \tr\{M_{ik}W^{[t]} \} \;\forall
      \{i,k\} \in \mathcal{E}_s
      \\
      &\underline{P}_i \leq P_{it} \leq \bar{P}_i, \;\;
      \underline{Q}_i \leq Q_{it} \leq \bar{Q}_i, \; \forall
      i\in\N^{[t]},
      \\
      &\underline{V}_{i}\leq \tr\{M_{i}W^{[t]} \} \leq \bar{V}_{i},
      \;\forall i\in \N^{[t]},
      \\
      & \tr\{M_{ik}W^{[t]} \} \leq \bar{V}_{ik}, \;\forall \{i,k\} \in
      \mathcal{E}^{[t]},
      \\
      & P_{it} = \tr\{Y_{i}^{[t]}W^{[t]}\} + P_{D_i}+
      \sum\nolimits_{k\in\N^{[t]}_{i,s}}\tr\{\hat{Y}_{ik}U_{ik} \},
      \\
      &Q_{it} = \tr\{\bar{Y}_{i}^{[t]}W^{[t]}\} + Q_{D_i}+
      \sum\nolimits_{k\in\N^{[t]}_{i,s}}\tr\{\hat{\bar{Y}}_{ik}U_{ik}
      \},
    \end{alignat*}
  \end{subequations}
  where $\E^{[t]}$ is the set of connected lines for contingency
  $t\in\Co$. Each $W^{[t]}$ corresponds to the solution of the
  contingency $t\in\Co$. All the constraints involved in~\textbf{(P3)}
  are imposed on every $W^{[t]}$. Since there is no coupling between
  $W^{[t_1]}$ and $W^{[t_2]}$ for different $t_1,t_2\in\Co$, all the
  properties of the virtual-voltage approximation characterized above
  remain valid. The problem size of the security constrained OTS grows
  linearly with respect to $|\Co|$
}

\section{Partition-Based OTS Algorithm}\label{sec:propMIPsol}

\rev{The virtual-voltage approach described in
  Section~\ref{sec:bilinear} finds a candidate switching,
  cf.~\eqref{eq:alphas-from-P3} for the OTS problem. This, together
  with the solution to the convex relaxation \textbf{(P3)}, provide
  upper and lower bounds on the optimal value of the OTS problem, cf.
  Proposition~\ref{prop:reconstructed-sol}.  Here, we discuss how to
  refine the reconstructed solution to better approximate the solution
  of the original optimization problem.  One approach consists of
  using the branch-and-bound algorithm~\cite{IEG:02} and relying on
  \textbf{(P3)} to generate the required branch lower bounds in its
  execution. However, this approach can easily become intractable as
  the number 
  of switchable lines grows because of the large number of cases where
  \textbf{(P3)} must be executed. Instead,
we propose the partition-based OTS algorithm shown in
Algorithm~\ref{algo:partition_int}.
The method relies on graph
  partitioning and is not directly applicable to security constrained
  OTS formulations, cf. Section~\ref{sec:SCOTS}. 
}

\begin{algorithm}
  \caption{Partition-Based OTS Algorithm}\label{algo:partition_int}
\rev{
  \begin{algorithmic}[1]
    \State \textbf{Compute} the optimal solution $W^{\opt}$ of
    \textbf{(P3)}
    \State \textbf{Construct} graph reduction $\graph_r$
    (Section~\ref{sec:graph-reduction})
    \newline
    \Comment{\textit{Clusters together nodes connected by switchable
        lines}}
    \State \textbf{Assign}  adjacency matrix to $\graph_r$
    (Section~\ref{sec:g-partition})
    \newline
    \Comment{\textit{Sets weights according to edge influence on
        optimal solution}}
    \State \textbf{Compute} cut set $\E_c$ to partition $\graph_r$
    into $n$ subgraphs (Section~\ref{sec:g-partition})
    \newline
    \Comment{\textit{Partitions network graph into smaller components
        accounting for impact on optimal solution}}
    \State \textbf{Solve}  integer optimization problem \textbf{(P4)} on each
    subgraph to find $\alpha_p^{\opt}$
    (Section~\ref{sec:optimization-on-subgraphs})
    \newline
    \Comment{\textit{Solves subproblems of smaller size}}
    \State \textbf{Solve} \textbf{(P1)}-$\alpha_p^{\opt}$
    (Section~\ref{sec:optimization-fixed-topology})
    \newline
    \Comment{\textit{Reconstructs solution of original problem}}
  \end{algorithmic}
}
\end{algorithm}

\subsubsection{Graph reduction}\label{sec:graph-reduction}

This is a step prior to graph partitioning which is motivated by the
following observation. The graph partitioning should not result in
nodes connected by a switchable line belonging to different
subgraphs. This is because, if that were the case, then solving the
OPF associated with each subgraph cannot capture how the switch in
that specific line affects the optimal value of the original OTS
problem. To address this, we `hide' the nodes that are connected
by~$\E_s$ to the partitioning algorithm that finds the edge cut~$\E_c$
so as to ensure $\E_s\cap\E_c = \emptyset$.  Let $\N_s:=\setdef{i \in
  \N}{ \{i,k\}\in\E_s }$ and let $\N_{s,i}$ be the set of nodes that
are connected to node $i\in\N_s$ through a line in~$\E_s$. All nodes
in $\N_{s,i}$ are clustered as one representative node and all the
edges connected to one of $\N_{s,i}$ are considered being connected to
the representative node. This results in a graph $\graph_r =
\big((\N\setminus\N_s)\cup\V,\E_r\big)$, where $\V$ is the collection
of representative nodes. Notice that $\E_r\subseteq\E$ and $\E_r$ is a
strict subset of $\E$ if there is $\{i,k\}\in\E$ such that a path
connecting nodes $i$ and $k$ exists in the graph $(\N,\E_s)$.
Figure~\ref{fig:graph_s} illustrates the construction of~$\graph_r$
and has $\E_r\subset \E$ as one edge of $\E$ is dropped in the process
of graph reduction.

\begin{figure}
  \centering
  \includegraphics[width=0.35\textwidth]{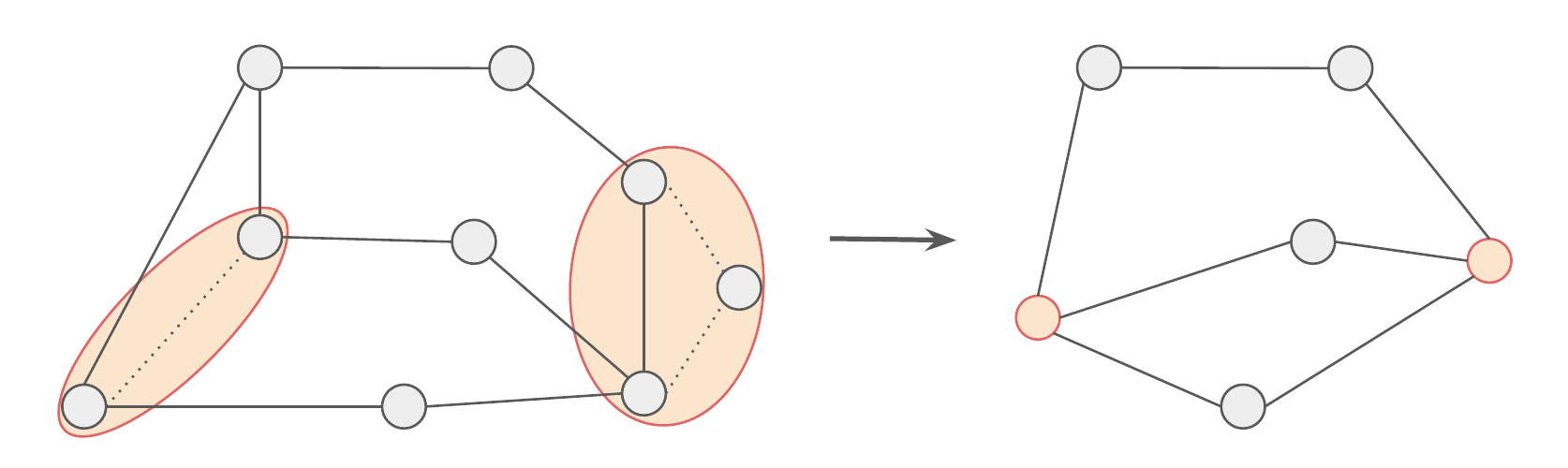}
  \caption{Graph reduction.  Nodes connected by~$\E_s$ are collapsed
    into one node. The dash lines denote edges in~$\E_s$; the solid
    lines denote the edges in~$\E$.}\label{fig:graph_s}
  \vspace*{-3ex}
\end{figure}

The graph reduction step described above only makes sense when not all
lines are switchable because otherwise, it results in a graph with a
single node.  For OTS scenarios where all lines are switchable, one
can either skip the graph reduction step, or pre-select a set of lines
that should remain non-switchable.
	
\subsubsection{Graph partitioning}\label{sec:g-partition}
Our next step is to find an edge cut set $\E_c$ of the graph
$\graph_r$.  In order to minimally affect the optimal value
$p^{\opt}$, the graph partitioning is based on the optimal dual
variables of \textbf{(P3)}.  The optimum dual variables measure how
the optimal value $p^{\opt}_3$ of \textbf{(P3)} changes with respect
to the corresponding constraint. Formally, by taking the derivative of
the Lagrangian of \textbf{(P3)}, we have the following for
$ i\in\N\setminus\N_s$,
\begin{align*}
  &\underline{\lambda}_i^{\opt_3} \hspace{-1mm}= \frac{\partial p^{\opt}_3}{\partial
    \underline{P}_i}, \; \bar{\lambda}_i^{\opt_3} \hspace{-1mm}= \frac{\partial
    p^{\opt}_3}{\partial \bar{P}_i}, \;
\underline{\gamma}_i^{\opt_3}\hspace{-1mm} = \frac{\partial p^{\opt}_3}{\partial
 	\underline{Q}_i}, \; \bar{\gamma}_i^{\opt_3} \hspace{-1mm}= \frac{\partial
 	p^{\opt}_3}{\partial \bar{Q}_i},
\end{align*}
and for $ i\in\V$,
\begin{align*}
  &\underline{\lambda}_i^{\opt_3} = \sum_{k \in \N_{s,i}}\frac{\partial p^{\opt}_3}{\partial
    \underline{P}_k}, \; \bar{\lambda}_i^{\opt_3} = \sum_{k \in \N_{s,i}}\frac{\partial
    p^{\opt}_3}{\partial \bar{P}_k}, \\
&\underline{\gamma}_i^{\opt_3} = \sum_{k \in \N_{s,i}}\frac{\partial p^{\opt}_3}{\partial
	\underline{Q}_k}, \; \bar{\gamma}_i^{\opt_3} = \sum_{k \in \N_{s,i}}\frac{\partial
	p^{\opt}_3}{\partial \bar{Q}_k},
\end{align*}
With this interpretation, we define edge weights as follows
\begin{align*}
  A(i,k)=
  \begin{cases}
    \underset{l
      \in\{i,k\}}{\sum}\hspace{-2mm}\bar\lambda_l^{\opt_3}\hspace{-1mm}
    + \underline\lambda_l^{\opt_3}\hspace{-1mm} +\bar\gamma_l^{\opt_3}
    \hspace{-1mm}+ \underline\gamma_l^{\opt_3}, & \{i,k\}\in\E_r,
    \\
    0, & \text{otherwise}.
  \end{cases}
\end{align*}
Given the adjacency matrix~$A$, we do an $n$-optimal partition on
$\graph_r$, which gives $\graph_r[\V^0_1],\cdots,\graph_r[\V^0_n]$
with $\cup_{i=1}^n\V^0_i = (\N\setminus\N_s)\cup\V$. Since all the
removed edges are in $\E$, we can use the same cut for the partition
of $\graph$: $\graph[\V_1],\cdots,\graph[\V_n]$ with
$\cup_{i=1}^n\V_i = \N$. Such partition ensures
$\E_c\cap \E_s = \emptyset$. The intuition is that the cut minimally
perturbs $p^{\opt}$ because it select edges with minimal weight for
the weighted graph~$(\graph,A)$.

Finding the optimal cut set is NP-hard. There are
algorithms~\cite{JPH:04b,AA-GK:06} that can approximate it in a few
seconds for graphs of the order of a thousand nodes. However, they do
not guarantee that the resulting subgraphs are connected.  To ensure
this property, we resort to spectral graph partitioning.

\begin{theorem}\longthmtitle{Fiedler's theorem of connectivity of
    spectral graph partitions}
  The two subgraphs resulting from spectral graph partitioning on a
  connected graph are also connected.
\end{theorem}

The proof is available in~\cite[Corollary 2.9]{MF:75}.
To derive a $n$-partition, one can implement spectral graph
partitioning recursively~$n$ times.
%
%
Since we aim for subgraphs with similar size, each iteration applies
spectral graph partitioning on the subgraph with the largest number of
nodes. The most computationally expensive step in this process is the
eigenvector computation, which only takes linear time, or
$\mathcal{O}(n)$. This low complexity is reflected on the negligible
computational time in our simulations.  \rev{ Even though this
  recursive spectral partitioning does not in general lead to an
  $n$-optimal partition, we can characterize a lower bound for the sum of weights 
%
%
 for
  each iteration $l$ of the recursive partitioning by
\begin{align}\label{eq:LB_partition}
  \frac{1}{2} C_{\opt,l}^2 \leq \sum_{\{i,k\}\in\E_{c,l}}A(i,k),
\end{align}
where $C_{\opt,l}$ is the optimal value for $2$-optimal partitioning
and $\E_{c,l}$ is the edge cut set in iteration~$l$ (without loss of
generality, we have assume $A$ is
normalized). Inequality~\eqref{eq:LB_partition} follows directly from
Cheeger inequalities~\cite{FC:07} and the Courant-Fischer
Theorem~\cite{BM-YA-GC-OO:91}.
%
%
}

\subsubsection{Integer optimization on
  subgraphs}\label{sec:optimization-on-subgraphs}
Given a $n$-partition $\{\graph[\V_l]\}_{l=1}^n$, we define an
optimization problem associated with each subgraph. This problem is a
variant of \textbf{(P2)} that is convenient for the reconstruction of
the solution of \textbf{(P2)} over the original~$\graph$. For subgraph
$l$, let $\E_l$ be its set of edges, $W_l\in\Sn_+^{|\V_l|}$ the
decision variable, $\hat{\E}_{s,l}$ the set of switchable lines, and
$\B_l$ the set of nodes in $\V_l$ that connects to at least one node
of another subgraph.  Each subgraph $l$ solves
\begin{flalign*}
  \textbf{(P4)} \min_{\substack{W_l\succeq 0,
      \alpha_{ik}\in\{0,1\},\;
      \forall \{i,k\}\in\hat\E_{s,i}}} \sum\nolimits_{i \in \N_G\cap\V_l}
  \Big(c_{i2} P_i^2+ c_{i1} P_i\Big), &&
\end{flalign*}
subject to
\begin{subequations}\label{eq:conP4}
  \begin{alignat*}{6}
    & \underline{P}_i \leq P_i \leq \bar{P}_i, \quad \underline{Q}_i \leq Q_i \leq \bar{Q}_i , \; \forall i\in \V_l,
    \\
    & \underline{V}_i^2 \leq \tr\{M_iW_l\} \leq \bar{V}_i^2, \;
    \forall i\in \V_l , \displaybreak[0]
    \\
    & \tr\{M_{ik}W_l\} \leq \bar{V}_{ik}, \; \forall \{i,k\} \in \E_l.
    \\ 
    &\nonumber \text{For all }i\in\V_l\setminus\B_l,
    \\
    &P_i = \tr\{Y_iW_l\} + P_{D_i}+ \hspace{-1mm}\sum\nolimits_{k,
      \{i,k\}\in\E_{s,i}}\hspace{-3mm}\alpha_{ik}\tr\{Y_{ik}W_{l,ik}
    \},
    \\
    &Q_i = \tr\{\bar{Y}_iW_l\} + Q_{D_i}+ \hspace{-1mm}\sum\nolimits_{k,
      \{i,k\}\in\E_{s,i}}\hspace{-3mm}\alpha_{ik}\tr\{\bar{Y}_{ik}W_{l,ik}
    \}.  
    \\
    &\nonumber \text{For all }i\in\B_l,
    \\
    &P_i = \tr\{Y_iW_l\} + P_{D_i} + \Pa_{l,i}+ \hspace{-5mm}\sum_{k,
      \{i,k\}\in\E_{s,i}}\hspace{-5mm}\alpha_{ik}\tr\{Y_{ik}W_{l,ik}
    \}, 
    \\
    &Q_i = \tr\{\bar{Y}_iW_l\}\hspace{-.5mm} +\hspace{-.5mm} Q_{D_i} +
    \Q_{l,i}+ \hspace{-5mm}\sum_{k,
      \{i,k\}\in\E_{s,i}}\hspace{-5mm}\alpha_{ik}\tr\{\bar{Y}_{ik}W_{l,ik}
    \}, 
  \end{alignat*}
\end{subequations}
where $\Pa_{l,i} = \sum_{k\in\N\setminus\V_l,\{i,k\}\in\E}
P_{ik}^{\opt_3}$ sums the active power flow from the solution of
\textbf{(P3)}, $\Q_{l,i}$ is defined similarly, and with a slight
abuse of notation, all $M_i$, $M_{ik}$, $Y_i$, $\bar{Y}_i$, $Y_{ik}$
take proper dimensions matching $W_l$. Notice that \textbf{(P4)} is
still NP-hard due to $\alpha_{ik}$, but the number of switches
$|\E_{s,i}|$ in each partition is less than $|\E_s|$, and decreases
with~$n$.  The addition of $\Pa_{l,i}$ and $\Q_{l,i}$ in \textbf{(P4)}
accounts for the coupling between $\graph[\V_l]$ and the other
subgraphs. For each subgraph, these terms are constant and do not
provide an exact approximation of the power exchanged between the
subgraphs -- since they do not take into account the dependency of the
terminal voltage on the solutions determined on the other
subgraphs. Therefore, putting together the solutions obtained for each
subgraph may not result in a feasible solution of \textbf{(P2)}, but
rather a solution to \textbf{(P2)} with a perturbation
on~\eqref{eq:const-c}. We address this next.

\subsubsection{Full SDP optimization with fixed
  topology}\label{sec:optimization-fixed-topology}
In the last step, we define the candidate optimal switch
$\alpha_{p}^{\opt}\in\{0, 1\}^{|\E_s|}$ from the solutions of
\textbf{(P4)} obtained in the previous step.  With this in place, we
solve \textbf{(P1)}-$\alpha^{\opt}_p$ to obtain the candidate optimal
solution~$W^{\opt}_p$ and output $(\alpha^{\opt}_p, W^{\opt}_p)$ as
the reconstructed solution of~\textbf{(P2)}.


\vspace*{-1ex}
\section{Simulation Studies}\label{sec:sim}

In this section we illustrate the performance of the virtual-voltage
approximation and the partition-based OTS algorithm on standard IEEE
test systems. All simulations are done on a desktop with 3.5GHz CPU
and 16GB RAM, using MATLAB and its CVX toolbox~\cite{MG-SB:14-cvx} to
solve the convex optimization problems.  In all our tables except the last row in Table~\ref{tab:BARON}, ``lower
bound'' refers to the optimal value of \textbf{(P3)} and ``upper
bound'' refers to the optimal value of \textbf{(P1)}-$\alpha$, where
$\alpha$ is determined by the corresponding method.


\rev{ 
\subsection{Comparison with McCormick Relaxation}
We implement the virtual-voltage approximation and compare its
performance against the solution obtained from \textbf{(P2)} with the
McCormick approximation (cf.  Remark~\ref{rem:McC}). For the latter,
we use two different estimates on the upper bounds of the line power
flows. In one case, we use the conservative bounds $\bar{P}_{ik}=
-\underline{P}_{ik} = 5$(p.u.) and $\bar{Q}_{ik} =
-\underline{Q}_{ik}= 5$(p.u.) for all $\{i,k\}\in\hat{\E}_s$. These
bounds come from the heuristic estimation on the largest line
active/reactive power flow. In the other case, we set the tighter
bounds $\bar{P}_{ik}= -\underline{P}_{ik} = 1$(p.u.)  and
$\bar{Q}_{ik} = -\underline{Q}_{ik}= 0.5$(p.u.) for all
$\{i,k\}\in\hat{\E}_s$, based on our knowledge of the solution of the
IEEE test nominal case.  We include an additional cost function on the
line power losses to promote optimal solutions with some edges
disconnected.  For each test case, a set of switchable lines are
selected by the following policy: given a design parameter $p \in
\Pint$, we rank all lines by the norm of their admittance in ascending
order and select the first $p$ as the set of switchable lines.
The rationale for this selection is that each line with small
admittance places a small correlation between its terminal nodes and,
as a result, they are more likely to be disconnected. The selection of
$p$ is based on the size of the network.

\begin{table*}[htb!]
  \begin{center}
    \begin{tabular}{c c | c |c |c |c |c |c |}
    	\cline{3-8}
    	&& IEEE 30 & IEEE 39 & IEEE 57 & IEEE 89 & IEEE 118 & IEEE 300
    	\\
    	\cline{2-8}
    	& \multicolumn{1}{|c|}{$\#$ of switches} & 5 & 5 & 5 & 30 & 40 & 40 
    	\\
    	\hline
      \multicolumn{1}{|c}{
        \multirow{2}{*}{
          Virtual-voltage approximation}}&
      \multicolumn{1}{|c|}{lower bound}  & 1265 (1)
       & 135003 (1) & 50912 (1) & 179325 (1) & 151594 (3) & 1086369 (1)
      \\
      \cline{3-8} \multicolumn{1}{ |c  }{}
      &
      \multicolumn{1}{|c|}{upper bound}
      & 1270 (1) & 135303 (1) & 52267 (2) & 188896 (1) & 152707 (1) & 1096117 (1)
      \\
      \hline
      \multicolumn{1}{|c}{\multirow{2}{*}{
          McCormick relaxation w/ $5$(p.u.) bounds}}&
      \multicolumn{1}{|c|}{lower bound}  & 1046 (1) & 133591 (4) & 50209 (1) & 177643 (1) & 146466 (1) & 1054130 (2)
      \\
      \cline{3-8} \multicolumn{1}{ |c  }{}
      &       \multicolumn{1}{|c|}{upper bound}
      & 4654 (3) & 135660 (1) & 52267 (2) & 230355 (1) & 153477 (1) & NaN
      \\ 
      \hline
      \multicolumn{1}{|c}{
        \multirow{2}{*}{McCormick relaxation w/ $1$(p.u.) bounds} } &
      \multicolumn{1}{|c|}{lower bound}  &  1046 (1)  &  133591 (1) & 50209 (1) & NaN & 147919 (1) & NaN
      \\
      \cline{3-8} \multicolumn{1}{ |c  }{}
      &       \multicolumn{1}{|c|}{upper bound}
      & NaN & 135303 (1) & 67667 (2) & NaN & 152484 (1) & NaN
      \\ 
      \hline
    \end{tabular}
  \end{center}
  \caption{Performance of the virtual-voltage  and the
    McCormick approximation. The integer value in parentheses
    is the rank of the SDP solutions.  
    ``NaN'' means  that CVX cannot find a feasible solution (even
    though the problem may still be feasible, see~\cite{MG-SB-YY:11}
    for details).}\label{tab:GivenSw} 
  \vspace*{-1.5ex}
\end{table*}
The virtual-voltage approach yields discrete variables close to
$\{0,1\}$ for all the test cases and, in contrast, the McCormick
relaxation has most of them around~$0.5$. Table~\ref{tab:GivenSw}
shows the values obtained by both approximations, and confirms that
the virtual-voltage approach gives better solutions than the McCormick
relaxation. It is worthwhile to note that every virtual voltage matrix
$U$ has its condition number comfortably more than 1000 for all the
test cases. The smallest condition number is 1492. This validates the
statements of Lemma~\ref{lem:rank1-U} and Remark~\ref{rem:rank1-U} for
a physically meaningful virtual voltage.

In both the McCormick relaxation and virtual-voltage methods, the
number of decision variables grows linearly with respect
to~$p$. Though $p$ is relatively small compared to the dimension of
$W$ and is not the main factor for optimization complexity, we observe
in Table~\ref{tab:McC-VV} that for a given $p$, the McCormick
relaxation returns a solution noticeably faster than the
virtual-voltage method. The reason is that the McCormick relaxation
does not introduce LMIs for the switchable lines. In this regard, the
virtual-voltage approach trades time complexity for better-quality
solutions.
\begin{table*}[htb!]
  \begin{center}
    \begin{tabular}{c c | c |c |c |c | c| c|}
      \cline{3-8}
      && IEEE 30 & IEEE 39 & IEEE 57 & IEEE 89 & IEEE 118 & IEEE 300
      \\
      \cline{2-8}
      & \multicolumn{1}{|c|}{$\#$ of switches} & 5 & 5 & 5 & 30 & 40 & 300
      \\
      \hline
      \multicolumn{1}{|c}{
        \multirow{2}{*}{
          Virtual-voltage approximation}}&
      \multicolumn{1}{|c|}{lower bound}  & 7.78
      & 9.55 & 18.72 & 59.01 & 149.37  & 2668
      \\
      \cline{3-8} \multicolumn{1}{ |c  }{}
      &
      \multicolumn{1}{|c|}{upper bound}
      & 5.52 & 6.71 & 12.63 & 39.05  & 88.55 & 2356
      \\
      \hline
      \multicolumn{1}{|c}{\multirow{2}{*}{
          McC. relaxation w/ $5$(p.u.) bounds}}&
      \multicolumn{1}{|c|}{lower bound}  & 4.93 & 7.69 & 12.92 & 57.19 & 99.19 & 2394
      \\
      \cline{3-8} \multicolumn{1}{ |c  }{}
      &       \multicolumn{1}{|c|}{upper bound}
      & 5.1 & 5.8 & 12.62 & 44.48 & 88.28 & 2250
      \\ 
      \hline
      \multicolumn{1}{|c}{
        \multirow{2}{*}{McC. relaxation w/ $1$(p.u.) bounds} } &
      \multicolumn{1}{|c|}{lower bound}  &  6.24  &  5.8 & 13.34 & 94.51 & 104.18 & 2338
      \\
      \cline{3-8} \multicolumn{1}{ |c  }{}
      &       \multicolumn{1}{|c|}{upper bound}
      & 5.82 & 6.61 & 9.87 & 42.25 & 87.88 & 2311
      \\ 
      \hline
    \end{tabular}
  \end{center}
  \caption{Comparison of the computational time between McCormick
    and virtual-voltage approximations.}\label{tab:McC-VV} 
\vspace*{-1.5ex}
\end{table*}

\subsection{Comparison with BARON} 
We compare the virtual-voltage approximation with the general purpose
mixed-integer nonlinear programming solver BARON~\cite{BARON} in several
small-scale IEEE examples. With the default settings and maximal
computational time at 1000 seconds, BARON does not always return a
feasible solution, see Table~\ref{tab:BARON}.  {\small
  \begin{table*}[htb!]
    \begin{center}
      \begin{tabular}{c c|c|c|c|c|c|c|}
        \cline{3-8}
        & & \multicolumn{2}{|c|}{IEEE 30} & \multicolumn{2}{|c|}{IEEE 39} & \multicolumn{2}{|c|}{IEEE 57} \\ \hline
        \multicolumn{2}{|c|}{$\#$ of switchable lines} & 5 & 10 & 1 & 5 & 1 & 5 \\ \hline
        \multicolumn{1}{|c|}{\multirow{2}{*}{Virtual-voltage approx.}} & lower bound & 1265  & 1265 & 135295 & 135003 & 50997 & 50912 \\	\cline{2-8}
        \multicolumn{1}{ |c|  }{} & upper bound & 1269.9 & 1269.9 & 135303 & 135303 & 52267  & 52267  \\	\hline
        \multicolumn{1}{ |c|  }{\multirow{2}{*}{BARON}} & lower bound & $-2.2\cdot 10^7$ & $-2.3\cdot 10^7$ & $-2.7\cdot 10^8$ & $-2.6\cdot 10^8$ & $-6.8\cdot 10^7$ & $-6.9\cdot 10^7$ \\ \cline{2-8} \multicolumn{1}{ |c|  }{} & upper bound & 1269.9 & 1269.9 & 135303 & $3.4\cdot 10^{8}$ & $7.6\cdot 10^{8}$ & $7.6\cdot 10^{8}$ \\	\hline
      \end{tabular}
    \end{center}
    \caption{Comparison between solutions of BARON and virtual-voltage
      approximation. We pre-select 5 switchable lines for all the test
      cases.}\label{tab:BARON} 
  \vspace*{-1.5ex}
  \end{table*}
} We observe that BARON gives a feasible solution only when it finds
one in the preprocessing stage, which happens for the IEEE 30 bus test
case and the IEEE 39 bus test case with one switchable line. The
branch-and-bound process made in BARON only slowly increases the lower
bound, and barely refines the upper bound given in the preprocessing
stage. On top of discrete binary variables for switching, OTS involves
nonlinearities coming from bilinear products. All these factors can
prevent general purpose solvers such as BARON from finding the optimal
solution efficiently. On the other hand, the virtual-voltage
approximation is built on provably accurate SDP relaxations of OPF and
makes use of the specific structure of the OTS problem.
Table~\ref{tab:BARON} clearly shows the difference in performance
between both methods.

\subsection{Partition-based OTS algorithm}

We examine here the performance of the partition-based
algorithm. Table~\ref{tab:allsw} shows the result of implementing on
the IEEE 118 and IEEE 300 bus test cases the following methods: (i)
the virtual-voltage approximation with 40 switchable lines; and (ii)
the partition-based OTS algorithm with $40$ switchable lines; and
(iii) the virtual-voltage approximation with all the lines
switchable. For the case with 40 switchable lines, one can see that
the partition-based OTS algorithm refines the virtual-voltage
approximation. The solutions obtained by the virtual-voltage
approximation and the partition-based OTS algorithm with 40 switchable
lines are both close to the lower bound. In addition, SDP returns a
rank-1 solution for every case with 40 switchable lines. We quantify
their accuracy by upper bounding the percentage error with the true
optimal solution using $(p^{\opt}_{1} -
p^{\opt}_{3,\text{all}})/{p^{\opt}_{3,\text{all}}}\cdot100\%$.

We simulate case (iii) where all the lines are switchable to compare
the results with the QC relaxation method~\cite{HH-CC-PVH:17}. The QC
relaxation method takes 10 hours to solve the OTS problem for both
IEEE 118 and 300 bus test cases employing servers with 4334 CPUs and
64GB memory. The QC relaxation method converges to a near optimum with
around $1\%$ error for the IEEE 118 test case, while it is unable to
provide a feasible solution for the IEEE 300 bus case. In contrast,
solving \textbf{(P3)} with all lines switchable takes significantly
less time for both cases (around 2min30sec and 1h30min, respectively)
as shown in Table~\ref{tab:allsw}. Though we only retrieve a rank-2
solution from \textbf{(P3)} for IEEE 118 bus case, the condition
number of the decision variable $W$ is larger than 250, which
translates to an error (or violation of constraints) less than
$1\%$. Therefore, the virtual-voltage approximation leads to a
solution of a similar quality as QC. The reason of not retrieving a
rank-1 solution with the virtual-voltage approximation is that it
provides high-rank solutions when the $U$ matrix is introduced on
every line (this is one of the reasons for pre-selecting a subset of
lines). 

\begin{table*}[htb]
  \begin{center}
    \begin{tabular}{c c | c | c | c | c | c | c |}
      \cline{3-8}
      &&\multicolumn{2}{|c|}{Optimal values} &
      \multicolumn{2}{|c|}{Bound on errors} &
      \multicolumn{2}{|c|}{Computation times (sec)}
      \\
      \cline{3-8}
      && IEEE 118 & IEEE 300 & IEEE 118 & IEEE 300 & IEEE 118 & IEEE 300
      \\
      \hline
      \multicolumn{1}{|c}{
        \multirow{2}{*}{
          Virtual-voltage approximation w/ all lines switchable}}&
      \multicolumn{1}{|c|}{lower bound}  & 
      150791 ($\bar{10}$) & 1090397 ($\bar{10}$) & \multirow{2}{*}{$0.36\%$} & \multirow{2}{*}{N/A} & \multirow{2}{*}{142.57} & \multirow{2}{*}{4859.62}
      \\
      \cline{2-4} \multicolumn{1}{ |c  }{}
      &
      \multicolumn{1}{|c|}{upper bound}
      & 151340 (2) & NaN   &  &  & &
      \\
      \hline
      \multicolumn{1}{|c}{      \multirow{2}{*}{
          Virtual-voltage approximation w/ 40 switchable lines}}&
      \multicolumn{1}{|c|}{lower bound}  & 151594 (1) & 1086369 (1) & \multirow{2}{*}{$1.27\%$} & \multirow{2}{*}{$0.52\%$}  & \multirow{2}{*}{149.37}
      & \multirow{2}{*}{5023.90}
      \\
      \cline{2-4} \multicolumn{1}{ |c  }{}
      &       \multicolumn{1}{|c|}{upper bound}
      & 152707 (1) & 1096117 (1) & & & &
      \\ 
      \hline
      \multicolumn{1}{|c}{
        Partition-based OTS w/ 40  switchable lines} & & 152505 (1) &
      1092967 (1) & $1.14\%$ & $0.24\%$ & 173 & 5055.88
      \\
      \hline
    \end{tabular}
  \end{center}
  \caption{Performance of the virtual-voltage approximation and the partition-based
    OTS algorithm on the IEEE 118 and IEEE 300 bus test cases. $(\bar{10})$ indicates that the rank of the decision variable is larger than 10.}\label{tab:allsw} 
  \vspace*{-3ex}
\end{table*}


\subsection{OTS with Security Constraints}
We also validate the virtual-voltage approximation method on OTS with
security constraints. We simulate both IEEE 118 and 300 bus test
systems with one contingency scenario and 40 pre-selected switchable
lines. Table~\ref{tab:scopf} shows the results for the IEEE 118 bus
test system with security constraint on line $\{106,107\}$. The
combined computational time in computing the lower and upper bounds
raises to 834 (395+438) seconds, which is longer than the one without
security constraint (149.37 seconds in Table~\ref{tab:allsw}). For the
IEEE 300 bus test case, CVX cannot find a feasible solution even when
we increase the maximal iterations to $600$ (the default is $100$
iterations), which took around 3.7 hours.
The drastic increase on the computational time limits the
applicability of the SDP-based approach to OTS 
with security constraints. 
\begin{table}[htb]
  \begin{center}
    \begin{tabular}{|c|c | c|}
      \hline
      & Optimal Value & Time (sec) 
      \\
      \hline 
      lower bound & 156896 (5) & 396 \\
      \hline 
      upper bound & 158660 (1) & 438
      \\
      \hline
    \end{tabular}
  \end{center}
  \caption{Performance of the virtual-voltage approximation for OTS
    with security constraint on line $\{106,107\}$.}\label{tab:scopf} 
  \vspace*{-3ex}
\end{table}

}

\vspace*{-1ex}
\section{Conclusions}
\rev{ We have considered OTS problems with security constraints. For
  these scenarios, the standard SDP relaxation of the problem remains
  non-convex because of the presence of bilinear terms and the
  discrete variables.  We have proposed an approximation based on the
  introduction of virtual-voltages to convexify the bilinear terms.
  We have also characterized several of its properties regarding
  physical interpretation and lower and upper bounds on the optimal
  value of the original problem.  To handle the presence of the
  discrete variables, we have built on the virtual-voltage
  approximation to propose a graph partition-based algorithm that
  significantly reduces the computational complexity of solving the
  original problem.  The high degree of accuracy and the reduction in
  computational complexity observed in simulations makes the proposed
  algorithms promising for OTS applications. Future work will
  incorporate other types of discrete controls,
and investigate distributed methods for general mixed-integer OPF problems.
}
\vspace*{-1ex}	
\bibliographystyle{IEEEtran}
\bibliography{alias,Main-add,JC}

\begin{thebibliography}{10}
\providecommand{\url}[1]{#1}
\csname url@samestyle\endcsname
\providecommand{\newblock}{\relax}
\providecommand{\bibinfo}[2]{#2}
\providecommand{\BIBentrySTDinterwordspacing}{\spaceskip=0pt\relax}
\providecommand{\BIBentryALTinterwordstretchfactor}{4}
\providecommand{\BIBentryALTinterwordspacing}{\spaceskip=\fontdimen2\font plus
\BIBentryALTinterwordstretchfactor\fontdimen3\font minus
  \fontdimen4\font\relax}
\providecommand{\BIBforeignlanguage}[2]{{%
\expandafter\ifx\csname l@#1\endcsname\relax
\typeout{** WARNING: IEEEtran.bst: No hyphenation pattern has been}%
\typeout{** loaded for the language `#1'. Using the pattern for}%
\typeout{** the default language instead.}%
\else
\language=\csname l@#1\endcsname
\fi
#2}}
\providecommand{\BIBdecl}{\relax}
\BIBdecl

\bibitem{CYC-SM-JC:17-allerton}
C.-Y. Chang, S.~Mart{\'\i}nez, and J.~Cort\'es, ``Convex relaxation for
  mixed-integer optimal power flow problems,'' in \emph{Allerton Conf.\ on
  Communications, Control and Computing}, Monticello, IL, 2017, pp. 307--314.

\bibitem{JM-ME-RA:99}
J.~Momoh, M.~El-Hawary, and R.~Adapa, ``A review of selected optimal power flow
  literature to 1993. {P}art {I}: Nonlinear and quadratic programming
  approaches,'' \emph{IEEE Transactions on Power Systems}, vol.~14, no.~1, pp.
  96--104, 1999.

\bibitem{SF-IS-SR:12}
S.~Frank, I.~Steponavice, and S.~Rebennack, ``Optimal power flow: a
  bibliographic survey {I},'' \emph{Energy Systems}, vol.~3, no.~3, pp.
  221--258, 2012.

\bibitem{HA-SDB-MLS:17}
H.~Abdi, S.~D. Beigvand, and M.~L. Scala, ``A review of optimal power flow
  studies applied to smart grids and microgrids,'' \emph{Renewable and
  Sustainable Energy Reviews}, vol.~71, pp. 742--766, 2017.

\bibitem{MRA-MEE:09}
M.~R. AlRashidi and M.~E. El-Hawary, ``A survey of particle swarm optimization
  applications in electric power systems,'' \emph{IEEE Transactions on
  Evolutionary Computation}, vol.~13, no.~4, pp. 913--918, 2009.

\bibitem{PEOY-JMR-CAC:08}
P.~E.~O. Yumbla, J.~M. Ramirez, and C.~A. Coello, ``Optimal power flow subject
  to security constraints solved with a particle swarm optimizer,'' \emph{IEEE
  Transactions on Power Systems}, vol.~23, no.~1, pp. 33--40, 2008.

\bibitem{AGB-PNB-CEZ-VP:02}
A.~G. Bakirtzis, P.~N. Biskas, C.~E. Zoumas, and V.~Petridis, ``Optimal power
  flow by enhanced genetic algorithm,'' \emph{IEEE Transactions on Power
  Systems}, vol.~17, no.~2, pp. 229--236, 2002.

\bibitem{KWH-SSO-RPO:11}
K.~W. Hedman, S.~S. Oren, and R.~P. O'Neill, ``A review of transmission
  switching and network topology optimization,'' in \emph{IEEE Power and Energy
  Society General Meeting}, Detroit, MI, 2011, electronic proceedings.

\bibitem{JGR-LJBM:99}
J.~G. Rolim and L.~J.~B. Machado, ``A study of the use of corrective switching
  in transmission systems,'' \emph{IEEE Transactions on Power Systems},
  vol.~14, no.~1, pp. 336--341, 1999.

\bibitem{JDF-RR-AC:12}
J.~D. Fuller, R.~Ramasra, and A.~Cha, ``Fast heuristics for transmission-line
  switching,'' \emph{IEEE Transactions on Power Systems}, vol.~27, no.~3, pp.
  1377--1386, 2012.

\bibitem{PAR-JMF-AR-MCC:12}
P.~A. Ruiz, J.~M. Foster, A.~Rudkevich, and M.~C. Caramanis, ``Tractable
  transmission topology control using sensitivity analysis,'' \emph{IEEE
  Transactions on Power Systems}, vol.~27, no.~3, pp. 1550--1559, 2012.

\bibitem{SF-JL-AA:17}
S.~Fattahi, J.~Lavaei, and A.~Atamt{\"u}rk, ``A bound strengthening method for
  optimal transmission switching in power systems,'' \emph{arXiv preprint
  arXiv:1711.10428}, 2017.

\bibitem{TP-KWH:12}
T.~Potluri and K.~W. Hedman, ``Impacts of topology control on the {ACOPF},'' in
  \emph{IEEE Power and Energy Society General Meeting}, San Diego, CA, 2012,
  electronic proceedings.

\bibitem{MS-JDF:14}
M.~Soroush and J.~D. Fuller, ``Accuracies of optimal transmission switching
  heuristics based on {DCOPF} and {ACOPF},'' \emph{IEEE Transactions on Power
  Systems}, vol.~29, no.~2, pp. 924--932, 2014.

\bibitem{HH-CC-PVH:17}
H.~Hijazi, C.~Coffrin, and P.~{Van~Hentenryck}, ``Convex quadratic relaxations
  for mixed-integer nonlinear programs in power systems,'' \emph{Mathematical
  Programming Computation}, vol.~9, no.~3, pp. 321--367, 2017.

\bibitem{JM-MM-JCV:16}
J.~Mare{\v{c}}ek, M.~Mevissen, and J.~C. Villumsen, ``{MINLP} in transmission
  expansion planning,'' in \emph{Power Systems Computation Conference}, Genoa,
  Italy, 2016, electronic proceedings.

\bibitem{EB-SA-FP:17}
E.~Briglia, S.~Alaggia, and F.~Paganini, ``Distribution network management
  based on optimal power flow: Integration of discrete decision variables,'' in
  \emph{Annual Conference on Information Systems and Sciences}, Baltimore, MD,
  2017, electronic proceedings.

\bibitem{GPM:76}
G.~P. McCormick, ``Computability of global solutions to factorable nonconvex
  programs: Part {I} -- convex underestimating problems,'' \emph{Mathematical
  programming}, vol.~10, no.~1, pp. 147--175, 1976.

\bibitem{FB-JC-SM:08cor}
F.~Bullo, J.~Cort{\'e}s, and S.~Mart{\'\i}nez, \emph{Distributed Control of
  Robotic Networks}, ser. Applied Mathematics Series.\hskip 1em plus 0.5em
  minus 0.4em\relax Princeton University Press, 2009, electronically available
  at \url{http://coordinationbook.info}.

\bibitem{JL-SHL:12}
J.~Lavaei and S.~H. Low, ``Zero duality gap in optimal power flow problem,''
  \emph{IEEE Transactions on Power Systems}, vol.~27, no.~1, pp. 92--107, 2012.

\bibitem{RM-MA-JL:16}
R.~Madani, M.~Ashraphijuo, and J.~Lavaei, ``Promises of conic relaxation for
  contingency-constrained optimal power flow problem,'' \emph{IEEE Transactions
  on Power Systems}, vol.~31, no.~2, pp. 1297--1307, 2016.

\bibitem{RAJ-RS-BCP:12}
R.~A. Jabr, R.~Singh, and B.~C. Pal, ``Minimum loss network reconfiguration
  using mixed-integer convex programming,'' \emph{IEEE Transactions on Power
  Systems}, vol.~27, no.~2, pp. 1106--1115, 2012.

\bibitem{PAR-AR-MCC-EG-EN-CRP:12}
P.~A. Ruiz, A.~Rudkevich, M.~C. Caramanis, E.~Goldis, E.~Ntakou, and C.~R.
  Philbrick, ``Reduced {MIP} formulation for transmission topology control,''
  in \emph{Allerton Conf.\ on Communications, Control and Computing},
  Monticello, IL, 2012, pp. 1073--1079.

\bibitem{AJW-BFW:12}
A.~J. Wood and B.~F. Wollenberg, \emph{Power generation, operation, and
  control}.\hskip 1em plus 0.5em minus 0.4em\relax John Wiley \& Sons, 2012.

\bibitem{EBF-RPO-MCF:08}
E.~B. Fisher, R.~P. O'Neill, and M.~C. Ferris, ``Optimal transmission
  switching,'' \emph{IEEE Transactions on Power Systems}, vol.~23, no.~3, pp.
  1346--1355, 2008.

\bibitem{KWH-RPO-EBF-SSO:09}
K.~W. Hedman, R.~P. O'Neill, E.~B. Fisher, and S.~S. Oren, ``Optimal
  transmission switching with contingency analysis,'' \emph{IEEE Transactions
  on Power Systems}, vol.~24, no.~3, pp. 1577--1586, 2009.

\bibitem{MK-HG-MD:13}
M.~Khanabadi, H.~Ghasemi, and M.~Doostizadeh, ``Optimal transmission switching
  considering voltage security and {N}-1 contingency analysis,'' \emph{IEEE
  Transactions on Power Systems}, vol.~28, no.~1, pp. 542--550, 2013.

\bibitem{IEG:02}
I.~E. Grossmann, ``Review of nonlinear mixed-integer and disjunctive
  programming techniques,'' \emph{Optimization and Engineering}, vol.~3, no.~3,
  pp. 227--252, 2002.

\bibitem{JPH:04b}
J.~P. Hespanha, ``An efficient {M}atlab algorithm for graph partitioning,''
  University of California, Santa Barbara, Tech. Rep., 2004.

\bibitem{AA-GK:06}
A.~Abou-Rjeili and G.~Karypis, ``Multilevel algorithms for partitioning
  power-law graphs,'' in \emph{International Parallel and Distributed
  Processing Symposium (IPDPS)}, Rhodes Island, Greece, 2006, electronic
  proceedings.

\bibitem{MF:75}
M.~Fiedler, ``A property of eigenvectors of nonnegative symmetric matrices and
  its application to graph theory,'' \emph{Czechoslovak Mathematical Journal},
  vol.~25, no.~4, pp. 619--633, 1975.

\bibitem{FC:07}
F.~Chung, ``Four {C}heeger-type inequalities for graph partitioning
  algorithms,'' \emph{Proceedings of ICCM}, pp. 751--772, 2007.

\bibitem{BM-YA-GC-OO:91}
B.~Mohar, Y.~Alavi, G.~Chartrand, and O.~Oellermann, ``The {L}aplacian spectrum
  of graphs,'' \emph{Graph theory, combinatorics, and applications}, vol.~2,
  no. 871-898, p.~12, 1991.

\bibitem{MG-SB:14-cvx}
M.~Grant and S.~Boyd, ``{CVX}: Matlab software for disciplined convex
  programming, version 2.1,'' Mar. 2014, available at
  \url{http://cvxr.com/cvx}.

\bibitem{MG-SB-YY:11}
M.~Grant, S.~Boyd, and Y.~Ye, ``{CVX} users' guide,'' 2009.

\bibitem{BARON}
{The Optimization Firm}, ``{BARON},'' \url{https://minlp.com/baron}.

\end{thebibliography}

\end{document}